\documentclass[11pt]{article}

\usepackage[utf8]{inputenc}

\usepackage{geometry}
\usepackage{graphicx}

\usepackage{booktabs}
\usepackage{array}
\usepackage{paralist}
\usepackage{verbatim}
\usepackage{subfig}
\usepackage{amsmath,amssymb, amsthm}
\usepackage{tikz-cd}
\usepackage{hyperref}

\usepackage{fancyhdr} 
\usepackage{sectsty}
\allsectionsfont{\rmfamily\mdseries\upshape} 

\usepackage[nottoc,notlof,notlot]{tocbibind} 
\usepackage[titles,subfigure]{tocloft}

\newcommand{ \cM }{ \mathcal M } 
 
\newcommand{ \cC }{ \mathcal C } 
 
\newcommand{ \cF }{ \mathcal F } 
 
\newcommand{ \cH }{ \mathcal H } 
\newcommand{ \cO }{ \mathcal O }

\newcommand{ \cD }{ \mathcal D }

\newtheorem{theorem}{Theorem}[section]
\newtheorem{lemma}[theorem]{Lemma}
\newtheorem{claim}[theorem]{Claim}
\newtheorem{definition}[theorem]{Definition}
\newtheorem{remark}[theorem]{Remark}

\usepackage{setspace}
\begin{document}

\begin{titlepage}
  \begin{center}
    \setlength{\parindent}{1.5em}
      \Large
      \textbf{An Alternative Model for Coherent Sheaves over Noetherian Schemes}

      \vspace{0.5cm}
      \Large  
      \textbf{Ron Held}
       
      \vspace{0.3cm}           
      \Large
      Tel-Aviv University \\
      August 2021
      \vspace{1.0cm}           

  \end{center}

  \begin{abstract}
    The category of coherent sheaves over a noetherian scheme is very important for studying the properties of a given scheme. For noetherian schemes it is a well-known fact that the topology can be fully recovered from the corresponding partially ordered set of its points together with the specializaion partial order.\\
    \indent
    Here we show that the same thing holds for the category of coherent sheaves over a noetherian scheme, i.e. there exists an equivalent description of this category in the "language" of this poset. First we build the equivalent notion of the structure sheaf, then we introduce the desired functor from the category of coherent sheaves to a certain category of presheaves over it, show this functor is fully faithful, describe its essential image and hence find an equivalent category for the category of coherent sheaves over a noetherian scheme.\\
    \indent
    This paper is based on the author's master's thesis.    
  \end{abstract}

\end{titlepage}

\renewcommand*\contentsname{Table of Contents}

\pagebreak

\tableofcontents

\pagebreak

\section{Introduction}
\indent
The goal of this paper is to give an alternative description for the important category of coherent sheaves over a noetherian scheme. This description is motivated by the fact that every noetherian scheme is a sober space, hence the topology of a noetherian scheme is completely determined by the topological relation between its points known as specialization: $x \to y$ whenever $y \in \bar{\{x\}}$. This property is an important distinction between the underlying spaces of schemes and algebraic varieties. Following this route we can consider a noetherian scheme with a structure sheaf $(X, \cO_{X})$ equivalently as a poset with a presheaf of local algebras $(\cC, \cO_{\cC})$. Using this structure we show in Section \ref{sec:specializaion} that each sheaf of $\cO_{X}$-modules gives rise to a presheaf of $\cO_{\cC}$-modules, using its stalks and localization homomorphisms between them. And so we introduce a specialization functor $S \colon \cM(\cO_{X}) \to PSh(\cO_{\cC})$ from the category of coherent sheaves of $\cO_{X}$-modules to the category of presheaves of $\cO_{\cC}$-modules.\\
\\
\indent
In Section \ref{sec:fullyfiathful} we prove that the functor $S$ is fully faithful (Claim \ref{Claim 1}). The proof is quite standard following this Key Lemma (Lemma \ref{Key Lemma}):\\
 Let $R$ be a noetherian ring, $M$ a finitely generated $R$-module. Then $M \cong \varprojlim_{p \in SpecR} M_{p}$.\\
\\
\indent
The proof of this lemma is given in Section \ref{sec:keylemma}, and utilizes the very special properties of the associated primes of a module. As our description is using the specialization relation between points, it turns out that all the "action" happens on the stalks at the associated primes. This property is given by the fact that the map $\nu \colon M \to \prod_{p \in Ass_{R}(M)}M_{p}$ which maps each element to its localizations at the associated primes is injective. This property together with the finiteness of the set of the associated primes provide us with a constructive way to reconstruct the elements of the module $M$ (i.e. global sections) from the elements of its localizations (i.e. germs). This procedure can be generalized naturally to general noetherian schemes and coherent sheaves over them via the notion of associated points of a coherent sheaf.\\
\\
\indent
Now that we have a fully faithful functor $S \colon \cM(\cO_{X}) \to PSh(\cO_{\cC})$, we want to find a description for the essential image of $S$, and so to find the full subcategory of $PSh(\cO_{\cC})$ which is equivalent to $\cM(\cO_{X})$ by $S$. This is done in Section \ref{sec:image}, where we define some notions of special (ringed) posets - "closed" (and open), inspired by the properties of sober spaces; and "affine noetherian", inspired by the result of the Key Lemma (Lemma \ref{Key Lemma}), bearing in mind a noetherian ring is a finitely generated module over itself. Looking at the objects of $PSh(\cO_{\cC})$ we construct a functor for noetherian affine posets $T \colon PSh(\cO_{\cC}) \to PSh(\cO_{\cC})$ by taking inverse limit and then localizing again. The main requirement for an object of $PSh(\cO_{\cC})$ to be in the essential image of $S$ would be that $T$ preserves the original structure of the presheaf. We call those elements of $PSh(\cO_{\cC})$ preserved by $T$ "quasi-admissible" presheaves. If in addition they have some finiteness property we call them "admissible". In Claim \ref{Claim 2} we prove that the essential image of $S$ is exactly the full subcategory of admissible presheaves. Denoting this subcategory by $\cM(\cO_{\cC})$ and combining the results of Claim \ref{Claim 1} and Claim \ref{Claim 2} we yield the following theorem and the main result of this paper (Theorem \ref{Theorem 1}): the specialization functor $S$ is an equivalence of categories $\cM(\cO_{X}) \cong \cM(\cO_{\cC})$.

\section{The Specialization Functor $S$}
\label{sec:specializaion}
\indent
Let $(X,\cO_{X})$ be a noetherian scheme. It is well-known that $X$ is a sober space, i.e. every irreducible closed subset of $X$ is the closure of exactly one point of $X$. Thus, we have a partial order on the points of $X$ called the "specialization" partial-order: $x \leq y$ when $y \in \bar{\{x\}}$. This partial order makes $X$ into a poset, which we will consider as a (small) category, denoted by $\cC$, whose objects are the points of $X$, and morphisms $x \to y$ whenever $x \leq y$ (i.e. $y \in \bar{\{x\}}$).\\
\indent
For every point $x \in X$, we have $\cO_{X,x} = \cO_{x}$ is a local algebra. For any regular function $f \in \cO_{X}(X)$ on $X$ (and equivalently on any other open $U \subset X$), denote by $f(x)$ its germ (i.e. image) in $\cO_{x}$. Whenever $x \leq y$ we have the canonical localization homomorphism of local algebras $o_{x,y} \colon \cO_{y} \to \cO_{x}$ given by: $f(y) \mapsto f(x) = f(y) \otimes 1$ ($f(y) \in \cO_{y}$). So, $\cO_{X}$ induces a natural presheaf (i.e. contravariant functor) of local algebras on $\cC$, which we will denote by $\cO_{\cC}$ ($\cO_{\cC}(x) := \cO_{X,x}$ and $\cO_{\cC}(x \to y) := o_{x,y}$).
\begin{definition}
\normalfont (i) A ringed poset $(\cC,\cO_{\cC})$ is a poset $\cC$ (considered as a small category) together with a presheaf of rings $\cO_{\cC}$ on $\cC$.\\
(ii) A locally ringed poset $(\cC,\cO_{\cC})$ is a poset $\cC$ (considered as a small category) together with a presheaf of local rings $\cO_{\cC}$ on $\cC$.
\end{definition}
\begin{remark}
\normalfont A ringed poset is actually a contravariant functor from the poset $\cC$ to the category of rings. A locally ringed poset is actually a functor from the poset $\cC$ to the category of local rings. A morphism of locally ringed spaces ($F \colon (\cC,\cO_{\cC}) \to (\cD,\cO_{\cD})$) is a morphism of functors, with the additional condition that any homomorphism $\cO_{\cC}(x) \to \cO_{\cD}(F(x))$ is a local homomorphism \cite[pp. 72-73]{hartshorne:ag}.
\end{remark}
\indent
Thus, the specialization partial order makes a noetherian scheme $(X,\cO_{X})$ into a locally ringed poset $(\cC,\cO_{\cC})$.
\begin{definition}
\normalfont Denote by $PSh(\cO_{\cC})$ the category of presheaves (contravariant functors) of modules over the presheaf of local rings (algebras) $\cO_{\cC}$.
\end{definition}
\indent
Let $\cF$ be a coherent sheaf of $\cO_{X}$-modules. For every $x \in X$ we have the stalk $\cF_{x}$ is an $\cO_{x}$-module (and a finitely generated one). Whenever $x \leq y$ we have a canonical homomorphism of stalks: $f_{x,y} \colon \cF_{y} \to \cF_{x}$ given by: $s_{y} \mapsto s_{x}$ for each $s_{y} \in \cF_{y}$, where $f_{x,y}$ is given by "localization": $s_{y} \mapsto s_{y} \otimes_{\cO_{y}} 1$. So every coherent sheaf $\cF$ induces a presheaf of (finitely-generated) modules over the presheaf of local algebras $\cO_{\cC}$, i.e. an object of $PSh(\cO_{\cC})$. Given a morphism of coherent sheaves $\psi \colon \cF \to \cH$, $\psi$ induces a morphism of presheaves over $\cO_{\cC}$ in a natural way (take the induced morphisms on the stalks). By the definition of $PSh(\cO_{\cC})$ a morphism $\phi$ of presheaves is compatible with "localization", i.e. $\phi_{x} = \phi_{y} \otimes 1$, whenever $x \to y$. \\
\\
\indent
To conclude, we have a "specialization" functor $S \colon \cM(\cO_{X}) \to PSh(\cO_{\cC})$, which assigns to every coherent sheaf of $\cO_{X}$-modules a presheaf of $\cO_{\cC}$-modules.

\section{The Functor $S$ is Fully Faithful}
\label{sec:fullyfiathful}
\begin{claim}
\label{Claim 1} 
The functor $S$ is fully faithful.\\
i.e.: Let $\cF, \cH$ be two coherent sheaves of $\cO_{X}$-modules, and let $(\phi_{x})_{x \in \cC}$ be a morphism of presheaves between $S\cF$ and $S\cH$. Then there exists a unique morphism of coherent sheaves $\psi \colon \cF \to \cH$ inducing this family.
\end{claim}
For the proof of Claim \ref{Claim 1} we will use the following Key Lemma (Lemma \ref{Key Lemma}), which we prove in the next section:\\
\\
Key Lemma: Let $R$ be a noetherian ring, $M$ a finitely generated $R$-module. Then $M \cong \varprojlim_{p \in SpecR} M_{p}$.\\

\begin{proof}[Proof of Claim \ref{Claim 1}, given Key Lemma (Lemma \ref{Key Lemma})]
As $\mathcal{H}om(\cF,\cH)$ is a coherent sheaf itself \cite[Ex. 5.1.6(b), p. 172]{liu:curves}, and $Hom(\cF,\cH) = \Gamma(X,\mathcal{H}om(\cF,\cH))$, it is enough to prove this statement for the case where $\cF = \cO_{X}$.\\
\indent
Suppose there exists such a morphism of coherent sheaves, then it is unique by the following reasoning: if there are two morphisms $\mu,\nu \colon \cF \to \cH$ such that $\mu_{x} = \nu_{x}$ for any $x \in X$, then we have $\mu = \nu$ \cite[p. 15]{eh:schemes}, i.e. two different morphisms between $\cF$ and $\cH$ induce two different morphisms of presheaves $S\cF \to S\cH$.\\
\indent
For the proof of existence, suppose first $X$ is an affine noetherian scheme, i.e. $X = SpecR$ where $R$ is a noetherian ring. The sheaf $\cH$ is a coherent sheaf, and so $\Gamma(X,\cH) = \cH(X)$ is a finitely generated $\cO_{X}(X)$-module and in particular $\cH_{x} = \cH(X)_{x}$ for every prime $x \in X$ \cite[Proposition 5.1(b), p. 110]{hartshorne:ag}. Let $(\phi_{x})_{x \in X}$ be a morphism of presheaves with $\phi_{x} \colon \cO_{\cC}(x) = \cO_{x} \to \cH_{x}$. Each morphism is completely determined by the choice of $\phi_{x}(1) = s_{x} \in \cH_{x}$, and so this is nothing but a family of germs on $X$ compatible with localization. Then we are done by Key Lemma (Lemma \ref{Key Lemma}).\\
\indent
Now, given an arbitrary noetherian scheme $X$, let $U_{i} = SpecR_{i}$ be a finite affine cover (each $R_{i}$ is noetherian). Again, let $(\phi_{x})_{x \in X}$ be a morphism of presheaves with $\phi_{x} \colon \cO_{\cC}(x) = \cO_{x} \to \cH_{x}$. Each morphism is completely determined by the choice of $\phi_{x}(1) = s_{x} \in \cH_{x}$, and so this is nothing but a family of germs on $X$ compatible with localization. So, let $(\sigma_{x})_{x \in X}$ be such a family of germs on $X$ compatible with localization. Each restriction of $\sigma$ to $U_{i}$ is again a family of compatible germs, and so we know it corresponds to a section of $\cH(U_{i})$. Glue those sections to gain a global section on $X$.\\
\indent
Thus, each family of compatible germs (on $X$) corresponds to a global section of $\cH$, and so for every $\cH \in \cM(\cO_{X})$ we have a bijection $Hom(\cO_{X},\cH) \leftrightarrow Hom(\cO_{\cC}, S\cH)$, and so $S$ is fully faithful.
\end{proof}

\section{Proof of Key Lemma}
\label{sec:keylemma}
In this section we prove the Key Lemma used for the proof of Claim \ref{Claim 1}:
\begin{lemma}
\label{Key Lemma}
Let $R$ be a noetherian ring, $M$ a finitely generated $R$-module. Then $M \cong \varprojlim_{p \in SpecR} M_{p}$.
\end{lemma}
For the proof we first provide some notations and mention essential properties regarding the associated points of a coherent sheaf and their connection to the associated primes of a module. These are standard notions in algebraic geometry [e.g. \cite[pp. 147-148]{mumford:red}] and commutative algebra [e.g. \cite[pp. 37-42]{matsumura:ring}].
\begin{definition}
\normalfont Let $X$ be a noetherian scheme, and let $\cF$ be a coherent sheaf of $\cO_{X}$-modules. An associated point of $\cF$ is a point $x \in X$ such that there exist an open neighborhood $x \in U \subset X$ and a section $s \in \cF(U)$ such that $x$ is a generic point of the support of s (where $Supp(s) := \{y \in U: s_{y} \neq 0\}$). In the case where $X=SpecR$ is affine, and so $\cF=\widetilde{M}$, we have the associated points of the coherent sheaf $\cF$ are exactly the associated primes of the $R$-module $M$.
\end{definition}
Thus, we shall mention now without a proof the definition and some properties of the associated primes of a module:
\begin{definition}
\normalfont Let $R$ be a (commutative) ring, and let $M$ be an $R$-module. An associated prime of $M$ is a prime ideal $p \in SpecR$ such that there exists $m \in M$ with $p = Ann(m)$. Denote the set of all the associated primes of $M$ by $Ass_{R}(M)$.\\
Equivalently, let $Z := Supp(M) = \{x \in SpecR \colon M_{x} \neq 0\}$. The subset $Z$ is closed, and so there is a finite number of points $x_{1},...x_{n} \in SpecR$ such that $Z = \bigcup_{i} \bar{\{x_{i}\}}$ (i.e. the $x_{i}'s$ are the generic points of the Support of $M$). Denote this set by $A_{M} := \{x_{1},...,x_{n}\}$. Now consider the set $Ass_{R}(M) := \bigcup_{N \subset M} A_{N}$ ($N$ is a submodule of $M$).\\
The two definitions coincide, but the second is more general in some sense.
\end{definition} 

The associated primes have some special and important properties:
\begin{claim}[properties of the associated primes of a module]
Let $R$ be a noetherian ring, and let $M$ be an $R$-module.
\begin{enumerate}
	\item $Ass_{R}(M) \subset Supp(M)$. \cite[Theorem 6.5 (ii), p. 39]{matsumura:ring}
	\item The set of minimal elements of $Ass_{R}(M)$ (minimal with respect to set-theoretic inclusion) and the set of minimal elements of $Supp(M)$ coincide. \cite[Theorem 6.5 (iii), p. 39]{matsumura:ring}
	\item Let $S \subset R$ be a multiplicatively closed subset. Then $Ass_{S^{-1}R}(S^{-1}M) = Ass_{R}(M) \cap \{p \in SpecR \colon p \cap S = \emptyset\}$. In particular, if $p \in SpecR$ any prime it follows that $Ass_{R_{p}}(M_{p}) = Ass_{R}(M) \cap Spec(R_{p})$. \cite[Theorem 6.2, p. 38]{matsumura:ring}
	\item If $M$ is a noetherian module (e.g. when $M$ is finitely generated and $R$ is noetherian) then $Ass_{R}(M)$ is finite. \cite[Theroem 6.5 (i), p. 39]{matsumura:ring}
	\item The map $\nu \colon M \to \prod_{p \in Ass_{R}(M)}M_{p}$ which maps each element to its localizations at the associated points is injective. \cite[5.5.3.1, p. 168]{vakil:rising}
\end{enumerate}
\end{claim}

Now we can turn to the proof of Key Lemma (Lemma \ref{Key Lemma}):

\begin{proof}
  Denote $X:=SpecR$.
  Let $(\phi_{x})_{x \in X}$ be a morphism of presheaves with $\phi_{x} \colon R_{x} \to M_{x}$. Each morphism is completely determined by the choice of $\phi_{x}(1) = s_{x} \in M_{x}$, and so this is nothing but a family of germs on $X$ compatible with localization. Recall \cite[Definition II.5, p. 110]{hartshorne:ag} that a section $s \in M$ is a function $s \colon X \to \bigsqcup_{x \in X} M_{x}$, such that for each $x \in X$ there is a neighborhood of $x$, $U_{x} \subset X$, and there are $\tau \in M$ and a regular function $f \in R$ such that for each $y \in U_{x}$ we have $f(y)$ is invertible (in $R_{y}$) and $s(y) = \frac{\tau}{f(y)} \in M_{y}$.\\
  \indent
  Let $\sigma:=(\sigma_{x})_{x \in X}$ be a family of germs on $X$ compatible with localization, and let $x \in X$ a point. We have $\sigma_{x} \in M_{x}$, i.e. there are $\tau \in M$ and a regular function $f \in R$ such that $f(x)$ is invertible and $\sigma_{x} = \frac{\tau}{f(x)}$.\\
  \indent
  As $M$ is a finitely generated $R$-module, and $R$ is a noetherian ring, it follows (property 4) that the set $T := Ass_{R}(M)$ is finite. Moreover, we have $T_{x} := Ass_{R_{x}}(M_{x}) = \{y \in T \colon x \in \bar{\{y\}}\} \subset T$ (property 3). In words: $T_{x}$ is the set of the associated points that have $x$ in their closure.\\ 
  \indent
  Let $Z_{x} := V(f) \bigcup (\cup_{y \in T \setminus T_{x}} \bar{\{y\}}))$. The subset $Z_{x}$ is closed (as $V(f)$ is closed and the union is finite) and $x \notin Z_{x}$, and so $U_{x} := X \setminus Z_{x}$ is a neighborhood of $x$. Now $U_{x}$ is an open subset with these two essential properties:
  \begin{enumerate}
    \item $f$ is invertible at every point $y \in U_{x}$ (i.e. $f(y)$ is invertible).
    \item The set of associated points contained in $U_{x}$ is exactly $T_{x}$ (i.e. we dropped all the "bad points" that do not have $x$ in their closure). So, For every $z \in U_{x}$ we have $T_{z} \subset T_{x}$.
  \end{enumerate}
  \indent
  Let $y \in U_{x}$ be a point such that $x \in \bar{\{y\}}$. By compatabilty of $\sigma$ with localization we have $\sigma_{y} = \sigma_{x} \otimes 1 = \frac{\tau}{f(y)} \in M_{y}$. This holds in particular when $y \in T_{x}$.\\
  \indent
  For every $z \in X$, denote by $\nu_{z} \colon M_{z} \to \prod_{y \in T_{z}} M_{y}$ (where $T_{z} := Ass_{R_{z}}(M_{z})$ as before) the map which maps any element of $M_{z}$ to its localizations at the associated points. $\nu_{z}$ is injective for every $z \in X$ (property 5).\\
  \indent
  Now let $z \in U_{x}$ be a point such that $x \notin \bar{\{z\}}$. We have $T_{z} \subset T_{x}$. By compatabilty of $\sigma$ with localization we have $\sigma_{z} \otimes 1 = \sigma_{y} = \frac{\tau}{f(y)}$ for each $y \in T_{z}$. But $\frac{\tau}{f(z)} \otimes1 = \frac{\tau}{f(y)}$ for each $y \in T_{z}$ either, and so $\nu_{z}(\sigma_{z}) = \nu_{z}(\frac{\tau}{f(z)})$, and since $\nu_{z}$ is injective we have $\sigma_{z} = \frac{\tau}{f(z)} \in M_{z}$.\\ 
  \indent
  So, for every $y \in U_{x}$ we found $\tau \in M$ and a regular function $f \in R$ such that $f(y)$ is invertible and $\sigma_{y} = \frac{\tau}{f(y)}$. And so for every $x \in X$ there is such a neighborhood $U_{x}$, and hence we have $(\sigma_{x})_{x \in X}$ is a section of $M$.
\end{proof}

\begin{remark}
\label{finite points remark}
\normalfont The key property we used in the proof was the embedding $\nu \colon M \to \prod_{p \in Ass_{R}(M)}M_{p}$, with the finiteness of $Ass_{R}(M)$. For the case of a general $R$-module $M$, if there exists a finite subset $W \subset SpecR$ such that there is an embedding $\nu_{W} \colon M \to \prod_{p \in W}M_{p}$, then a similar proof will show $M \cong \varprojlim_{p \in SpecR} M_{p}$. For example, this holds for the case where $SpecR$ is finite itself or when $R$ is a local or a semi-local ring.
\end{remark}

\section{The Essential Image of the Functor $S$}
\label{sec:image}
\indent
We proved (Claim \ref{Claim 1}) that $S$ is a fully faithful functor. Now we would like to describe the essential image of $S$, i.e. to find a (full) subcategory of $PSh(\cO_{\cC})$ equivalent to $\cM(\cO_{X})$ (with $S$ as equivalence). In other words, we are trying to formulate necessary and sufficient conditions for the objects of $PSh(\cO_{\cC})$, such that they would be in the essential image of $S$.\\
\\
\indent
First, consider the case where $X$ is an affine noetherian scheme, i.e. $X = SpecR$, where $R$ is a noetherian ring. From the proof of Claim \ref{Claim 1} we know that if an object $G \in PSh(\cO_{\cC})$ is in the essential image of $S$, then the original (coherent) sheaf $\cF$ can be recovered from $G$ by taking the inverse limit for every open $U \subset X$, i.e.: $\cF(U) :=  \varprojlim_{x \in U} G(x)$. The "global" limit object $\cF(X)$ has a natural structure of $R$-module, with $r \cdot (s_{x})_{x \in X} := (\frac{r}{1} \cdot s_{x})_{x \in X}$ for every $r \in R$, $(s_{x})_{x \in X} \in \cF(X)$. Then we can localize $\cF(X)$ at every point $x$ to have a presheaf over $\cO_{\cC}$ again. This way we construct a functor $T: PSh(\cO_{\cC}) \to PSh(\cO_{\cC})$, and a morphism of functors (presheaves in $PSh(\cO_{\cC})$). In this case being preserved by $T$ (i.e. this morphism is isomorphism) is a necessary requirement for the elements in the essential image of $S$.\\
\\
\indent
We shall use again the general notion of locally ringed (sub)posets defined in Section \ref{sec:specializaion}, and define some special locally ringed posets:

\begin{definition}[for a topology on a poset] 
\normalfont A subposet is called closed if its points are the finite union of upsets of a single point. A subposet is called open if its complement is closed. Equivalently, consider the topology on a poset where the irreducible closed subsets are upper sets of a point.
\end{definition}

\begin{remark}
\normalfont In the case where $\cC$ is the corresponding poset of some sober space $X$ (e.g. when $X$ is a noetherian scheme) its closed subposets correspond to the (topological) closed subsets of $X$.
\end{remark}

\begin{definition}
\normalfont A locally ringed poset $(\cC, \cO_{\cC})$ is called affine noetherian if $R_{\cC} := \varprojlim_{x \in \cC} \cO_{\cC}(x)$ is a noetherian ring, and $(\cC, \cO_{\cC})$ is the locally ringed poset induced by $R_{\cC}$ via specialization.
\end{definition}
\indent
In general, for a given locally ringed poset $(\cC, \cO_{\cC})$ one can determine whether it can be covered by a finite cover of open affine noetherian locally ringed subposets (using the topology aforementioned). In the case where $\cC$ is the corresponding poset of some noetherian scheme $X$, this requirement is necessarily satisfied. Note that given a covering $\{\cC_{i}\}$ we can construct for each element in the covering a functor $T_{i}$ following the same aforementioned procedure. Now we can formulate the following definition:

\begin{definition}
\normalfont (i) an object $G \in PSh(\cO_{\cC})$ is called "quasi-admissible" if there exists a covering of $\cC$ by open locally ringed subposets $\cC_{i}$ such that $\forall i, x \in \cC_{i}: T_{\cC_{i}}G(x) \cong G(x)$ (i.e. the corresponding morphism of functors is an isomorphism).\\
(ii) If in addition to that the covering consists of (open) affine noetherian locally ringed subposets and $\forall i: \varprojlim_{x \in \cC_{i}} G(x)$ is a finitely-generated $R_{\cC_{i}}$-module, we call $G$ "admissible" (where $R_{\cC_{i}}$ is the noetherian ring corresponding to $\cC_{i}$).
\end{definition}

\begin{claim}
\label{Claim 2}
The essential image of $S$ is exactly the full subcategory of admissible presheaves.
\end{claim}
\begin{proof}
Let $G \in PSh(\cO_{\cC})$ be an admissible presheaf. As the definition of admissibility is local, we can consider for simplicity the case where $(\cC, \cO_{\cC})$ is an affine noetherian locally ringed poset, i.e. is the corresponding locally ringed poset of a noetherian scheme $X=SpecR$ (via the specialization partial order).\\ 
\indent
Denote $M := \varprojlim_{x \in \cC} G(x)$. We know by the admissibility of $G$ that $M$ is a finitely generated $R$-module, and so $\widetilde{M} \in \cM(\cO_{X})$ is a coherent sheaf of $\cO_{X}$-modules. As $\widetilde{M}$ is coherent we know its stalk at every point $x$ is isomorphic to the localization of $M$ at $x$. But, by the (quasi-)admissibility of $G$ we know $M \otimes_{R} R_{x} \cong G(x)$, and so $S\widetilde{M} = G$, i.e. $G$ is in the essential image of $S$.\\ 
\indent
From Key Lemma (Lemma \ref{Key Lemma}) it is clear if $\cF$ is a coherent sheaf, it follows that $S\cF$ is an admissible presheaf.\\
\indent
The claim follows immediately from the definition of coherent sheaves \cite[p. 111]{hartshorne:ag} and Claim \ref{Claim 1}.
\end{proof}

\begin{definition}
\normalfont Denote by $\cM(\cO_{\cC})$ the full subcategory of the admissible presheaves in $PSh(\cO_{\cC})$.
\end{definition}
\indent
Thus we can conclude:

\begin{theorem}
\label{Theorem 1}
The specialization functor $S$ is an equivalence of categories $\cM(\cO_{X}) \cong \cM(\cO_{\cC})$.
\end{theorem}
\begin{proof}
\end{proof}

\renewcommand\refname{References} 
\bibliographystyle{alpha}
\bibliography{bib}

\begin{thebibliography}{Mum99}

\bibitem[EH00]{eh:schemes}
David Eisenbud and Joe Harris.
\newblock {\em The Geometry of Schemes}.
\newblock Springer New York, NY, 2000.

\bibitem[Har77]{hartshorne:ag}
Robin Hartshorne.
\newblock {\em Algebraic Geometry}.
\newblock Springer New York, NY, 1977.

\bibitem[Liu02]{liu:curves}
Qing Liu.
\newblock {\em Algebraic Geometry and Arithmetic Curves}.
\newblock Oxford University Press, 2002.

\bibitem[Mat89]{matsumura:ring}
Hideyuki Matsumura.
\newblock {\em Commutative Ring Theory}.
\newblock Cambridge Studies in Advanced Mathematics. Cambridge University
  Press, 1987 (1989).

\bibitem[Mum99]{mumford:red}
David Mumford.
\newblock {\em The Red Book of Varieties and Schemes}.
\newblock Springer-Verlag Berlin Heidelberg 1999, 1999.

\bibitem[Vak17]{vakil:rising}
Ravi Vakil.
\newblock The rising sea: Foundations of algebraic geometry.
\newblock math216.wordpress.com/2017-18-course, 2017.

\end{thebibliography}

\end{document}